\documentclass[a4paper,11pt, final, BCOR=5mm,
	pagesize,
	pdftex]{article}
\usepackage[utf8]{inputenc}
\usepackage{float}
\usepackage{subcaption}

\usepackage{amsmath}
\usepackage{amsfonts}
\usepackage{amssymb}
\usepackage{mathrsfs}

\usepackage{graphicx}
\usepackage{pgfplots}
\pgfplotsset{compat=1.8}
\usepackage{tikz}
\usepackage{setspace,tabularx} 
\usepackage[draft=false,babel,tracking=true,kerning=true,spacing=true]{microtype} 

\usepackage{color}
\usepackage{dsfont}
\usepackage{lipsum} 

\usepackage{fixltx2e} 
\usepackage[T1]{fontenc} 
\usepackage{lmodern} 

\numberwithin{equation}{section}

\newcommand{\R}{\mathbb{R}}
\newcommand{\N}{\mathbb{N}}

\newcommand{\1}{\mathbf{1}}
\newcommand{\Pro}{\ensuremath\mathbb{P}}
\newcommand{\norm}[1]{\left\lVert#1\right\rVert}
\newcommand{\sumj}{\sum_{j \leq d}}
\newcommand{\sumk}{\sum_{k>d}}
\newcommand{\sumn}{\sum_{i=1}^n}

\newcommand{\PV}{P_{\leq d}}

\newcommand{\PVH}{\hat{P}_{\leq d}}

\newcommand{\tr}{\textup{tr}}
\newcommand{\SG}{L}
\renewcommand{\leq}{\leqslant}
\renewcommand{\geq}{\geqslant}

\newcommand{\PD}{K}

\newcommand{\ld}{\lambda}

\usepackage{amsthm}

\newtheorem{theorem}{Theorem}
\newtheorem{lemma}[theorem]{Lemma}
\newtheorem{corollary}[theorem]{Corollary}

\newtheorem{assumption}[theorem]{Assumption}

\theoremstyle{remark}

\begin{document}
\title{High-probability bounds for the reconstruction error of PCA}
\author{Cassandra Milbradt\thanks{Humboldt-Universit\"at zu Berlin, Germany. E-mail: cassandra.milbradt@gmail.com}\qquad Martin Wahl\thanks{Humboldt-Universit\"at zu Berlin, Germany. E-mail: martin.wahl@math.hu-berlin.de\newline \textit{Key words and phrases.} principal component analysis, reconstruction error, oracle inequality, polynomial chaos\newline \textit{2010 Mathematics Subject Classification.} 62H25}}
\date{}


\maketitle

\begin{abstract}
We derive high-probability bounds for the reconstruction error of PCA in infinite dimensions. We apply our bounds in the case that the eigenvalues of the covariance operator satisfy polynomial or exponential upper bounds.
\end{abstract}

\section{Introduction and Notation}

Principal component analysis (PCA) is a standard tool for dimensionality reduction. Motivated by its extensions to functional PCA and kernel PCA, we are concerned with the statistical properties of PCA in infinite dimensions. In such scenarios, the eigenvalues of the covariance (resp. kernel) operator often decay at a polynomial or nearly exponential rate, which in turn implies that the minimal reconstruction error has a certain rate in the reconstruction dimension. In this paper, we investigate whether the latter rate is also achieved by the empirically chosen model.

We consider a random variable $X$ with values in a Hilbert space $\mathcal{H}$. In what follows, $\langle\cdot,\cdot\rangle$ denotes the inner product in $\mathcal{H}$, with $\norm{\cdot}$ being the corresponding norm. We suppose that $X$ is  centered and strongly square integrable, meaning that $\mathbb{E}X=0$ and $\mathbb{E}\| X\| ^2< \infty$. The goal is to reduce the dimensionality of $X$ by finding the minimizer $\PV$ of the reconstruction error $R(P)=\mathbb{E}\| X-PX\| ^2$ over the class $\mathcal{P}_d$ of orthogonal projections of rank $d$. Yet, the distribution of $X$ is unknown and therefore the minimizer $\PV$ cannot be computed. The idea is, for a given sequence $X_1, \dots, X_n$ of independent copies of $X$, to compute the minimizer $\PVH\in \mathcal{P}_d$ of the empirical reconstruction error $R_n(P)= n^{-1}\sumn \| X_i - PX_i\| ^2$. Given the interest in the performance of the empirically chosen model $\PVH$ especially when observing new data, it is natural to analyze the reconstruction error of PCA, i.e.~the random variable $R(\PVH)$.

Bounds for the reconstruction error of PCA and the corresponding excess risk $R(\PVH)-R(\PV)$ can be derived using the theory of empirical risk minimization. This has
been pursued by Shawe-Taylor, Williams, Cristianini and Kandola \cite{SWCK05} and Blanchard, Bousquet and Zwald \cite{BBZ07}. In \cite{SWCK05}, a slow $n^{-1/2}$-rate is derived, while
in \cite{BBZ07}, it is shown that the convergence rate of the excess risk can be faster than $n^{-1/2}$ (but depending on a spectral gap condition).
In Rei\ss{} and Wahl \cite{ReissWahl}, it is shown how a fast $n^{-1}$-rate can be obtained by a version of the well-known Davis--Kahan $\sin\Theta$ theorem.
Moreover, since the resulting bound breaks down quickly in the infinite-dimensional case, \cite{ReissWahl} and subsequently Jirak and Wahl \cite{MR4052188} and Wahl \cite{Wahl} developed new perturbation techniques, leading (among others) to sharp excess risk bounds in the case that the
eigenvalues of the covariance operator of $X$ have exponential or polynomial decay. 

Many algorithms in statistics and machine learning employ PCA as a first step to reduce the high dimensionality of the data. Analyzing such procedures, it often
turns out that the main interest is not in optimal rates for the excess risk but rather in bounds implying that empirical PCA produces a finite-dimensional model
having (up to a constant) the same reconstruction property as the optimal one (see, for example, Nouy \cite{Nouy}). Furthermore, the dependence on PCA is often
highly non-linear, hence expectation bounds, as derived in \cite{ReissWahl}, are not sufficient to analyze these procedures. Motivated by these two facts, our
main goal is to derive oracle inequalities in high-probability that can be expressed relative to the minimal reconstruction error. In particular, a major
achievement of our paper is to show that oracle inequalities of the type 
\begin{equation}\label{weakOracleBound}
R(\PVH)\leq C \min_{P \in \mathcal{P}_d} R(P)
\end{equation}
can be established (with high probability) under one-sided eigenvalue conditions. For instance, we apply our bounds to eigenvalues satisfying polynomial or exponential upper bounds,
which are key features in reproducing kernel Hilbert spaces and functional data approaches in machine learning and statistics, see
e.g.~\cite{BachJordan,MR2335249,pmlr-v75-belkin18a}.

We finish the introduction by describing the link of PCA to the spectral decomposition of the (empirical) covariance operator. The covariance operator of $X$ is denoted by $\Sigma=\mathbb{E}X\otimes X$. By the spectral theorem, there exists a sequence $\lambda_1\geq \lambda_2\geq\dots>0$
of  positive eigenvalues (which is summable since $\mathbb{E}\| X\| ^2< \infty$) together with an orthonormal system of eigenvectors $u_1,u_2,\dots$ such that $\Sigma=\sum_{j\ge 1}\lambda_j P_j$, with rank-one projectors $P_j=u_j\otimes u_j$, 
where $(u\otimes v)x=\langle v,x\rangle u$, for $u,v,x\in \mathcal{H}$. Without loss of generality we shall assume that the eigenvectors $u_1,u_2,\dots$ form an orthonormal basis of $\mathcal{H}$ such that $\sum_{j\ge 1}P_j=I$. In addition, we define the sample covariance of $X_1, \dots, X_n$ as $\hat{\Sigma} = n^{-1} \sum_{i = 1}^n X_i \otimes X_i$. Again, there exists a sequence $\hat{\ld}_1 \geq \hat{\ld}_2 \geq \cdots \geq 0$ of eigenvalues of $\hat{\Sigma}$ and an orthogonal basis of eigenvectors $\hat{u}_1, \hat{u}_2, \dots$ of $\mathcal{H}$ such that $\hat{\Sigma} = \sum_{j \geq 1} \hat{\ld}_j \hat{P}_j$ with $\hat{P}_j = \hat{u}_j \otimes \hat{u}_j$. Now, the reconstruction error and the empirical reconstruction error can be written as 
\[R(P) = \tr(\Sigma( I-P))\quad\text{and}\quad R_n(P) =\tr(\hat{\Sigma} (I-P)),
\] where $\tr(\cdot)$ denotes the trace. Hence, a minimizer of $R(P)$ maximizes $\tr (\Sigma P)$, from which one can deduce that a minimizer of the reconstruction error is given by the orthogonal projection onto the linear subspace spanned by the first $d$ eigenvectors of $\Sigma$, i.e.~\[
\PV = \sumj P_j\quad\text{satisfies}\quad \PV \in \arg\min_{P \in \mathcal{P}_d}R(P)\quad\text{with}\quad R(\PV) = \sumk \ld_k.
\]
Replacing the projections $P_j$ by $\hat{P}_j$, the minimizer $\PVH$ of the empirical reconstruction error is given in the same way.

\section{Main result}

In this section we formulate our main error bound for the reconstruction error of PCA. It relies on a sub-Gaussian assumption on the Karhunen--Lo\`eve coefficients $\eta_j=\lambda_j^{-1/2}\langle X,u_j\rangle$, $j\geq 1$. Extensions under weaker moment assumptions are possible but beyond the scope of this paper.
\begin{assumption}\label{AssSG}
 Suppose that the $(\eta_j)_{j \geq 1}$ are independent. Moreover, for some constant $\SG > 0$ suppose that
\begin{equation}
 \sup_{j\geq 1}\sup_{q\geq 1} q^{-1/2} \left(\mathbb{E}|\eta_j| ^q\right)^{1/q} \leq \SG.
\end{equation}
\end{assumption}
Assumption \ref{AssSG} is satisfied with $\SG=1$ if $X$ is Gaussian. It gives a slightly stronger notion of a sub-Gaussian random variable than the
one in e.g. \cite[Definition 5.7]{Vershynin} or \cite[Assumption 2.1]{ReissWahl}.  
\begin{theorem}\label{Theorem}
Under Assumption \ref{AssSG}, there are constants $c_1,c_2,C_1>0$ depending only on $\SG$ such that the following holds. Let $1\leq d'\leq d$ be natural numbers such that $\ld_{d'}\geq 2\ld_{d+1}$ and $\max(d',\ld_{d'}^{-1}\sumk \ld_k) \leq c_1 n$. Then, for all $1\leq t \leq c_2n$, with probability at least $1-\exp(-t)$,
\begin{equation}\label{EqTheorem1}
 R(\PVH) \leq \left( 1 + C_1 \left(\frac{d'}{n}+\frac{t}{n}\right)\right) \min_{P\in \mathcal{P}_{d'}}R(P).
\end{equation}
In particular, with probability at least $1-\exp(-c_2n)$,
\begin{equation}\label{EqTheorem2}
  R(\PVH) \leq (1+C_1(c_1+c_2))\min_{P \in \mathcal{P}_{d'}} R(P).
\end{equation}
\end{theorem}
The second bound \eqref{EqTheorem2} shows that the reconstruction error of PCA can be bounded by a constant (typically close to one) times the minimal reconstruction error over $\mathcal{P}_{d'}$.

\section{Examples and discussion}
\label{SecEx}
Let us illustrate our upper bound for eigenvalues satisfying a polynomial or nearly exponential upper bound. Such eigenvalue structures are typically considered
in the context of functional data or statistical machine learning, see e.g. \cite{BachJordan,MR15, pmlr-v75-belkin18a}. In these cases, $\min_{P \in \mathcal{P}_d} R(P)$ (equals the
remainder trace) has a certain rate in $d$ and our bounds imply that the same rate holds for $R(\PVH)$.
\begin{corollary}
\label{CorPD}
Suppose that Assumption \ref{AssSG} holds. Moreover, suppose that for some $\alpha>1$ and $\PD>0$ we have $\ld_j \leq \PD j^{-\alpha}$ for all $j\geq 1$. Then there are constants $c_1,c_2,C_1>0$ depending only on $\alpha,\PD$ and $L$ such that, for all $d\leq c_1 n$, with probability at least $1-\exp(-c_2n)$,
\begin{equation}\label{EqCorPD}
  R(\PVH) \leq  C_1 d^{1-\alpha}.
\end{equation}
\end{corollary}
If additionally $\ld_j \geq \PD^{-1} j^{-\alpha}$ for all $j\geq 1$, then we have $\min_{P\in \mathcal{P}_{d}}R(P)=\sum_{k>d}\lambda_k\geq c_3d^{1-\alpha}$ for a constant $c_3>0$, and \eqref{EqCorPD} can be reformulated as $R(\PVH) \leq C_1\min_{P\in \mathcal{P}_{d}}R(P)$. 
\begin{corollary}
\label{CorNED}
Suppose that Assumption \ref{AssSG} holds. Moreover, suppose that for $\alpha,\PD>0$ and $\beta\in (0,1]$ we have $\PD^{-1}\exp(-\alpha j^\beta)\leq \ld_j \leq \PD \exp(-\alpha j^\beta)$ for all $j\geq 1$. Then there are constants $c_1,c_2,C_1,C_2>0$ depending only on $\alpha,\beta,\PD$ and $\SG$ such that, for all $d\leq c_1 n$, with probability at least $1-\exp(-c_2n)$,
\[
R(\PVH) \leq  C_1 d^{1-\beta}\exp(-\alpha (d+1)^\beta)\leq C_2\min_{P \in \mathcal{P}_{d}} R(P).
\]
\end{corollary}
Let us compare these bounds to the Davis--Kahan $\sin \Theta$ theorem, one of the most well-known perturbation bounds for eigenspaces; see e.g. \cite{MR3371006} for a recent
statistical account. Applied to the excess risk (cf. Proposition 2.2 in \cite{ReissWahl}) it gives for every $1\leq t\leq n$, with probability at least
$1-\exp(-t)$,
\begin{equation}
\label{EqDK}
        R(\hat P_{\le d})- \min_{P\in \mathcal{P}_d} R(P) \leq  C\frac{t}{(\lambda_d-\lambda_{d+1})n},
\end{equation}
where $C >0$ is a constant depending only on $L$ and the trace of $\Sigma$. We observe that the upper bound typically increases for fixed sample size $n$ and increasing $d$, whereas the minimal reconstruction error decreases in $d$. For instance, 
if $\ld_j = \PD \exp(-\alpha j^\beta)$, $j\geq 1$, the bound \eqref{EqDK} breaks down for $d$ of size $(\log n)^{1/\beta}$. In contrast,  Corollary \ref{CorNED} does not depend on any gap condition and gives sharp bounds for $d\leq c_1n$.

\section{Proof of  Theorem \ref{Theorem}}
Throughout the proof, we use the letters $c,C>0$ for constants depending only on $\SG$ that may change from line to line. We start with formulating a more technical version of our main result in terms of the weighted covariance operator 
\[
\Sigma_{d'}:= S_{\leq d'}\Sigma S_{\leq d'}\quad \text{where} \quad S_{\leq d'} = \sum_{j\leq d'} (\ld_j - \ld_{d+1})^{-1/2}P_j  \quad\text{and}\quad d'\leq d.
\] 
Note that $S_{\leq d'}$ can be interpreted as the square-root of a (partial) reduced resolvent.
For this operator it holds that
\begin{align}\begin{aligned}
\| \Sigma_{d'}\| _{\infty} =& \frac{\ld_{d'}}{\ld_{d'} - \ld_{d+1}},\quad\tr(\Sigma_{d'}) = \sum_{j \leq d'} \frac{\ld_j}{\ld_j - \ld_{d+1}},\\
&\| \Sigma_{d'}\| ^2_{2} = \sum_{j \leq d'} \frac{\ld_j^2}{(\ld_j-\ld_{d+1})^2},
\end{aligned}\end{align}
where $\| \cdot\| _{\infty}$ and $\| \cdot\| _{2}$ denote the operator norm and the Hilbert--Schmidt norm, respectively. Using these quantities, we can express our main result as follows.
\begin{theorem}
\label{techThm}
Under Assumption \ref{AssSG}, there are constants $c_1',c_2',C_1'>0$ depending only on $\SG$ such that the following holds. For all natural numbers $d' \leq d$ satisfying
\begin{equation}
 \label{AssEV}
\| \Sigma_{d'}\| _{\infty} \tr(\Sigma_{d'})+\| \Sigma_{d'}\| _{\infty}\sumk \frac{\ld_k}{\ld_{d'}-\ld_k} \leq c_1'n,
\end{equation}
we have, for all $1\leq t\leq c_2'n/\| \Sigma_{d'}\| _{\infty}^2$, with probability at least $1-\exp(-t)$,
\begin{equation}
\label{EqtechThm}
R(\PVH) \leq \left( 1 + C_1'\left(\| \Sigma_{d'}\| _{\infty}\tr(\Sigma_{d'})\frac{1}{n}+\| \Sigma_{d'}\| ^2_{\infty} \frac{t}{n}\right)\right) \sum_{k > d'} \ld_k.
\end{equation}
\end{theorem}
Let us first show how  Theorem \ref{techThm} implies  Theorem \ref{Theorem}. First, Condition \eqref{AssEV} is implied by the assumptions of Theorem \ref{Theorem}, provided that $c_1=c_1'/8$, as can be seen by using that $\| \Sigma_{d'}\| _{\infty}\leq 2$ and $\tr(\Sigma_{d'}) \leq 2d'$. 
Since additionally $\sum_{k > d'}\ld_k = \min_{P\in \mathcal{P}_{d'}}R(P)$, \eqref{EqTheorem1} follows from \eqref{EqtechThm}. 
To prove Theorem \ref{techThm}, we will need three technical statements.
The following perturbation bound follows from the proofs of Proposition 3.5 and Theorem 2.12 in Rei\ss{} and Wahl \cite{ReissWahl} (see equations (3.9) and
(3.14) applied with $\mu = \ld_{d+1}$ and $r = s = d'$).
\begin{lemma}
\label{lem:boundRW}
Suppose that the assumptions of Theorem \ref{techThm} hold. Then, on the joint event
\begin{equation}
\label{EventRW}
 \{\| S_{\leq d'}(\Sigma - \hat{\Sigma})S_{\leq d'}\| _{\infty} \leq 1/4\} \cap \{\hat{\ld}_{d+1} - \ld_{d+1} \leq (\ld_{d'}- \ld_{d+1})/2\},
\end{equation}
we have
\[
 R(\PVH) \leq 16 \| \Sigma_{d'}\| _{\infty}\| S_{\leq d'}(\Sigma - \hat{\Sigma})P_{>d'}\| ^2_2 + \sum_{k>d'}\ld_k.
\]
\end{lemma}
The event in \eqref{EventRW} has been analyzed in Lemma 3.9 and Corollary 3.12 in Rei\ss{} and Wahl \cite{ReissWahl}, yielding the following lemma.
\begin{lemma}
\label{lem:pboundW}
Under the assumptions of Theorem \ref{techThm}, we have
{\small \begin{align*}
  \Pro&\left(\{\| S_{\leq d'}(\Sigma - \hat{\Sigma})S_{\leq d'}\| _{\infty} \leq 1/4\} \cap \{\hat{\ld}_{d+1} - \ld_{d+1} \leq (\ld_{d'}- \ld_{d+1})/2\}\right)\\
   &\hspace{7cm} \geq 1 - 2 \exp\left(-c_2'n/\| \Sigma_{d'}\| ^2_{\infty}\right).
\end{align*}}
\end{lemma}
The next step is to prove a concentration inequality for $\| S_{\leq d'} (\hat\Sigma - \Sigma)P_{>d'}\| _2$. It can be derived from a Hilbert space version of the Hanson--Wright inequality derived in Adamczak, Lata\l a and Meller 
\cite{ALM}. Alternatively, one can also multiply out the Hilbert--Schmidt norm, resulting in a polynomial chaos of degree
$4$ in sub-Gaussian random variables, and then apply Theorem 1.4 in \cite{AdamczakWolff}.
\begin{lemma}
\label{LemCI}
Under the assumptions of Theorem \ref{techThm}, we have, for all $t> 0$, with probability at least $1-2\exp(-t)$,
\begin{equation}
\label{EqLemCI}
 \| S_{\leq d'} (\hat\Sigma - \Sigma)P_{>d'}\| _2^2\leq  C \left( \tr(\Sigma_{d'})\frac{1}{n} + \| \Sigma_{d'}\| _{\infty} \frac{t}{n} + \| \Sigma_{d'}\| _{\infty} \frac{t^2}{n^2}\right)\sum_{k>d'}\ld_k.
\end{equation}
\end{lemma}

\begin{proof}
The main idea of the proof is that we can rewrite the left-hand side in \eqref{EqLemCI} in terms of the squared Hilbert--Schmidt norm of a quadratic form in the Karhunen--Lo{\`e}ve coefficients $\eta_{ij}:= \ld_j^{-1/2} \langle X_i, u_j \rangle$ with values being Hilbert-Schmidt operators. Indeed, we have
\[
S_{\leq d'} (\Sigma - \hat{\Sigma})P_{>d'} = \sum_{j\leq d'} \sum_{k>d'} \sum_{i=1}^n \biggl\{ \frac{1}{n}\sqrt{\frac{\ld_j\ld_k}{\ld_j -\ld_{d+1}}}u_j \otimes u_k \biggr\} \eta_{ij} \eta_{ik}.
\]
By Assumption \ref{AssSG}, the $\eta_{ij}$, $(i,j)\in \{1,\dots,n\}\times \N:= \tilde{\N}$, are independent, centered and sub-Gaussian random variables, meaning that we can apply Corollary 16
in \cite{ALM} (with the norm being the Hilbert--Schmidt norm) to the quadratic form
$\sum_{(i_1,j_1), (i_2, j_2) \in \tilde{\N}}a_{(i_1,j_1), (i_2,j_2)}\eta_{i_1j_1}\eta_{i_2j_2}$ with Hilbert--Schmidt operators
\begin{align*}
a_{(i_1,j_1), (i_2,j_2)} &= \frac{1}{2n} \left\{ \sqrt{\frac{\ld_{j_1}\ld_{j_2}}{\ld_{j_1} -\ld_{d+1}}} (u_{j_1} \otimes u_{j_2}) \1_{\{j_1\leq d',j_2>d'\}}\right.\\
&\left.\hspace{1.1cm} +  \sqrt{\frac{\ld_{j_2}\ld_{j_1}}{\ld_{j_2} -\ld_{d+1}}}(u_{j_2}\otimes u_{j_1})\1_{\{j_1>d',j_2\leq d'\}}\right\}\1_{\{i_1=i_2\}}.
\end{align*}
Note that, while  Corollary 16 in \cite{ALM} is formulated for finite index sets, it extends to our setting using e.g.~that the $a_{(i_1,j_1),(i_2,j_2)}$ are summable.
Therefore, we consider $x = (x_{(i,j)})_{(i,j) \in  \tilde{\N}}$ such that each component $x_{(i,j)} \in \R$. Similarly, we define $(x_{(i_1, j_1), (i_2, j_2)})_{(i_1, j_1), (i_2, j_2) \in \tilde{\N}}$.\newpage Exploiting the particular diagonal structure of
$(a_{(i_1,j_1),(i_2,j_2)})$, we compute, using the Cauchy--Schwarz inequality,
\begin{align*}
U_1 :=& \sup_{\| x\| _2\leq 1}\sqrt{\sum_{(i_1,j_1)}\Big\|\sum_{(i_2,j_2)}a_{(i_1,j_1),(i_2, j_2)} x_{(i_2,j_2)}\Big\|_2^2}\\
=& \, \frac{1}{2n}\sqrt{\max\Big\{\tr(\Sigma_{d'})\ld_{d'+1}\, , \, \| \Sigma_{d'}\| _{\infty} \sum_{k>d'}\ld_k\Big\}},\\
U_2 :=& \sup_{\big\| (x_{(i_1, j_1), (i_2, j_2)})\big\|_2\leq 1} \Big\| \sum_{(i_1,j_1), (i_2,j_2)} a_{(i_1,j_1), (i_2,j_2)} x_{(i_1,j_1), (i_2,j_2)}\Big\|_2\\
=&\, \frac{1}{\sqrt{2n}}\sqrt{\| \Sigma_{d'}\| _{\infty}\ld_{d'+1}},\\
V :=& \sup_{\| x\| _2\leq 1, \| y\| _2\leq 1} {\Big\| \sum_{(i_1,j_1), (i_2,j_2)} a_{(i_1,j_1), (i_2,j_2)} x_{(i_1,j_1)}y_{(i_2,j_2)}\Big\|}_2\\
 =&\, \frac{1}{\sqrt{2}}\frac{1}{n}\sqrt{\| \Sigma_{d'}\| _{\infty} \ld_{d'+1}},
\end{align*}
where $\| x\|_2^2 = \sum_{(i,j)\in \tilde{\N}} x_{(i,j)}^2$ and $\| (x_{(i_1, j_1), (i_2, j_2)})\| _2^2 = \sum_{(i_1, j_1), (i_2, j_2) \in \tilde{\N}} x_{(i_1, j_1), (i_2, j_2)}^2$. Applying Corollary 16 from \cite{ALM} with $U:= U_1+U_2$ and $t > 0$, we get
\begin{align}\label{HansonWright}
\begin{aligned}
\Pro &\biggl( \| S_{\leq d'}(\Sigma- \hat{\Sigma})P_{>d'}\| _2 \geq t+C\sqrt{n^{-1}\tr(\Sigma_{d'})\sum_{k>d'}\ld_k}\biggr)\\
&\hspace{5cm} \leq 2 \exp\left(-\frac{1}{C}\min\left\{\frac{t^2}{U^2}, \frac{t}{V}\right\}\right)
\end{aligned}
\end{align}
and thus, with probability at least $1-2\exp(-t)$,
\begin{align*}
&\| S_{\leq d'} (\Sigma- \hat{\Sigma})P_{>d'}\| _2^2\\
&\qquad \leq C \biggl(\sqrt{n^{-1}\tr(\Sigma_{d'})\sum_{k>d'}\ld_k}+ \max\left\{U\sqrt{t}\, ,\, Vt\right\}\biggr)^2\\
&\qquad \leq C\left(\tr(\Sigma_{d'})\frac{1}{n}+\tr(\Sigma_{d'}) \frac{t}{n^2} + \| \Sigma_{d'}\| _{\infty} \frac{t}{n}+ \| \Sigma_{d'}\| _{\infty}\frac{t^2}{n^2}\right) \sum_{k>d'}\ld_k.
\end{align*}
The second term can be dropped using \eqref{AssEV} yielding the claim.
\end{proof}

\begin{proof}[End of proof of Theorem \ref{techThm}]
We choose $0<t \leq c_2'\| \Sigma_{d'}\| _{\infty}^{-2} n$. Then $\| \Sigma_{d'}\| _{\infty} (t/n)^2$ in \eqref{EqLemCI} is dominated by $\| \Sigma_{d'}\| _{\infty} t/n$. Combining the results from Lemma \ref{lem:boundRW}--\ref{LemCI}, we can finally argue that \eqref{EqtechThm}
holds with probability of $1-4\exp(-t)$. Restricting  ourselves to $t\geq 1$, this inequality also holds with probability at least $1-\exp(-t)$ for all $1\leq t \leq c_2'\| \Sigma_{d'}\| _{\infty}^{-2}n$, provided that we adjust the constants $C_1'$ and $c_2'$ appropriately.
\end{proof}

\section{Proofs of the corollaries}
\label{AppProof}

\begin{proof}[Proof of Corollary \ref{CorPD}]
First, there is a constant $C>0$ depending only on $\alpha$ and $K$ such that $\sum_{k>d'}\lambda_k\leq C(d')^{1-\alpha}$ for every $d'\geq 1$. In what follows we let $c_1$ be the constant from Theorem \ref{Theorem} such that $d\leq c_1n$ and we assume that $C\geq 2K$. We now first consider 
the case that there is a $d'\leq d$ such that $\lambda_{d'}\geq Cd^{-\alpha}$. Then we have $\lambda_{d'}\geq 2\lambda_{d+1}$, as well as $\max(d',\ld_{d'}^{-1}\sumk \ld_k)\leq d\leq c_1n$. Hence, we can apply Theorem \ref{Theorem} to the maximal $d'\leq d$ with this property,
leading to $R(\PVH)\leq C_2\sum_{k>d'}\lambda_k< C_2Cd^{1-\alpha}+C_2\sum_{k>d}\lambda_k\leq 2C_2Cd^{1-\alpha}$ with probability at least $1-\exp(-c_2n)$ and some constant $C_2>0$ depending only on $L$. On the other hand, if $\lambda_{d'}<Cd^{-\alpha}$ for all $d'\leq d$ 
(implying that $\lambda_1<Cd^{-\alpha}$), then the trivial bound $R(\PVH)\leq \tr(\Sigma)$ leads to $R(\PVH)\leq Cd^{1-\alpha}+\sum_{k>d}\lambda_k\leq 2Cd^{1-\alpha}$. This completes the proof.
\end{proof}

\begin{proof}[Proof of Corollary \ref{CorNED}]
For all $d\geq 1$ we have
\begin{equation}
\label{EqCompRT}
C^{-1}(d+1)^{1-\beta}\exp(-\alpha (d+1)^\beta)\leq\sum_{k>d}\lambda_k\leq C d^{1-\beta}\exp(-\alpha d^\beta)
\end{equation}
with a constant $C>1$ depending only on $\alpha,\beta$ and $\PD$. In fact, this can be seen by a comparison of the sum with an integral combined with standard estimates for the incomplete Gamma function.  Moreover, by a concavity argument, we have 
$\beta x\leq 1-(1-x)^\beta\leq 2^{1-\beta}\beta x$, $x\in[0,1/2]$ and $\beta\in(0,1]$, from which we deduce that for $d'=d+1-k$ and $k\leq (d+1)/2$, the inequality
\begin{equation}
\label{EqCompEV}
\PD^{-2}\exp(\alpha\beta (d+1)^{\beta-1}k)\leq \frac{\lambda_{d'}}{\lambda_{d+1}}\leq\PD^2\exp(2^{1-\beta}\alpha\beta (d+1)^{\beta-1}k)
\end{equation}
holds. We now verify the conditions of Theorem \ref{Theorem} with $d'=d+1-k$, $k=\lceil (\alpha\beta)^{-1}(d+1)^{1-\beta}\log(2\PD^2)\rceil$. In the following, we will assume without loss of generality that $k\leq (d+1)/2$ (meaning that $d$ is larger than a constant depending on $\alpha,\beta$ and
$\PD$), because the bound \eqref{EqTheorem1} is obvious in the opposite case. First, using \eqref{EqCompEV}, $k$ is chosen such that $\ld_{d'}\geq 2\ld_{d+1}$ holds. In addition, using \eqref{EqCompRT}, we have $\ld_{d}^{-1}\sumk \ld_k \leq Cd^{1-\beta}$. Hence, the assumptions of Theorem
\ref{Theorem} are satisfied for $d\leq (c_1/C)n$, and \eqref{EqTheorem2} yields $ R(\PVH) \leq C_2 \sum_{k>d'}\lambda_k$ with probability at least $1-\exp(-c_2n)$. Thus the first claim follows from inserting  \eqref{EqCompRT} for $d=d'$ followed by an application of \eqref{EqCompEV}.
Moreover, the second claim follows from the first one, by applying \eqref{EqCompRT}.
\end{proof}

\subsection*{Acknowledgement}
The research of Martin Wahl has been partially funded by Deutsche Forschungsgemeinschaft (DFG) through grant CRC 1294 ``Data Assimilation'', Project (A4) ``Nonlinear statistical inverse problems with random observations''.

\bibliographystyle{plain}
\bibliography{lit}

\end{document}